\documentclass[a4paper]{amsart}
\usepackage{graphicx}
\usepackage{amssymb}
\usepackage{amsmath}
\usepackage{amsthm,amsfonts,bbm}
\usepackage{amscd}
\usepackage[all,2cell]{xy}

\UseAllTwocells \SilentMatrices

\newtheorem{thm}{Theorem}[section]

\newtheorem{cor}[thm]{Corollary}
\newtheorem{lem}[thm]{Lemma}

\newtheorem{prop-def}[thm]{Proposition-Definition}
\theoremstyle{definition}
\newtheorem{defi}[thm]{Definition}
\theoremstyle{remark}

\newtheorem{exm}[thm]{\bf Example}
\numberwithin{equation}{section}
\numberwithin{figure}{section}

\DeclareMathOperator{\sgn}{sgn}
\DeclareMathOperator{\Ori}{Ori}

\def\C{\mathbb{C}}

\def \ep{\epsilon}

\def \i{\mathbf{i}}

\def \lamax{\lambda_{\max}}
\def \lamin{\lambda_{\min}}

\def\R{\mathbb{R}}

\def\v{{\mathbf v}}

\begin{document}

\title[Largest H-eigenvalue of mixed graph]{On the largest eigenvalue of a mixed graph with partial orientation}

\author[B.-J. Yuan]{Bo-Jun Yuan}
\address{School of Mathematical Sciences, Anhui University, Hefei 230601, P. R. China}
\email{ybjmath@163.com}

\author[Y. Wang]{Yi Wang}
\address{School of Mathematical Sciences, Anhui University, Hefei 230601, P. R. China}
\email{wangy@ahu.edu.cn}

\author[Y.-Z. Fan]{Yi-Zheng Fan$^*$}
\address{School of Mathematical Sciences, Anhui University, Hefei 230601, P. R. China}
\email{fanyz@ahu.edu.cn}
\thanks{$^*$The corresponding author.
This work was supported by National Natural Science Foundation of China (Grant No. 11871073, 11771016).}

\date{\today}

\subjclass[2010]{05C50}

\keywords{Mixed graph; Hermitian adjacency matrix; largest eigenvalue; interlacing family; matching polynomial; partial orientation}

\date{}
\maketitle

\begin{abstract}
Let $G$ be a connected graph and let $T$ be a spanning tree of $G$.
A partial orientation $\sigma$ of $G$ respect to $T$ is an orientation of the edges of $G$ except those edges of $T$,
 the resulting graph associated with which is denoted by $G_T^\sigma$.
In this paper we prove that there exists a partial orientation $\sigma$ of $G$ respect to $T$
such that the largest eigenvalue of the Hermitian adjacency matrix of $G_T^\sigma$ is at most the largest absolute value of the  roots of the matching polynomial of $G$.
\end{abstract}

\section{Introduction}
Let $G=(V,E)$ be a simple graph.
Based on $G$, a {\it mixed graph} $D_G$ (or $D$ for short) is obtained from $G$ by orienting some of its edges,
 where $G$ is called the {\it underlying graph} of $D$, and the oriented edges are the \emph{arcs} of $D$.
Formally, a mixed graph $D$ is comprised of a vertex set $V(D)=V(G)$, a set of undirected edges and a set of arcs.
To avoid confusion, denote an undirected edge by $\{u,v\}$ and an arc by $(u,v)$.

Let $F$ be the set of edges of $G$ to be unoriented.
An \emph{orientation $\sigma$ of $G$ with respect to $F$} is defined as a skew-symmetric map:
$$ \sigma: (V(G) \times V(G)) \backslash F \to \{0, 1, -1\},$$
where $\sigma(u,v)=-\sigma(v,u)$, and $\sigma(u,v) \ne 0$ if and only if $\{u,v\} \in E(G)$.
We note here $F$ is considered as a set of ordered pairs of the vertices of $V(G)$ by replacing each element $\{u,v\}$ by $(u,v)$ and $(v,u)$.
The resulting mixed graph associated with $\sigma$ and $F$ is denoted by $G_F^\sigma$,
  where each edge in $F$ is unoriented (undirected), and each edge $e=\{u,v\}$ outside $F$ is oriented from $u$ to $v$ if $\sigma(u,v)=1$ or  from $v$ to $u$ otherwise.
If $F=\emptyset$, then $G_F^\sigma$, written as $G^\sigma$ in this case, is an \emph{oriented graph} or a mixed graph with \emph{complete orientation}.
If $F=E(G)$, then $G_F^\sigma$ is exactly the graph $G$ without any orientation.
Otherwise, $G_F^\sigma$ is called a mixed graph with \emph{partial orientation}.

The {\it Hermitian adjacency matrix} \cite{LL, mohar1} of $G_F^\sigma$
is defined to be a matrix $H(G_F^\sigma)=(h_{uv})$, where
$$h_{uv}= \left  \{
\begin{array}{ll}
1, & \hbox{~if~}  \{u,v\} \in F,\\
\i \sigma(u,v), & \hbox{~if~}  \{u,v\} \in E(G) \backslash F,\\
0, & \hbox{~otherwise~},
\end{array}
\right.
$$
where $\i=\sqrt{-1}$ is the imaginary unit.
Since $H(G_F^\sigma)$ is Hermitian, its eigenvalues are all real.
Denote by $\lamax(G_F^\sigma), \lamin(G_F^\sigma)$ the largest and the smallest eigenvalues of $H(G_F^\sigma)$ respectively.
The \emph{spectral radius} of $G_F^\sigma$, denoted by $\rho(G_F^\sigma)$, is defined to be the largest modulus of eigenvalues of $H(G_F^\sigma)$, which is equal to
$\max\{\lamax(G_F^\sigma), |\lamin(G_F^\sigma)|\}$.
In this paper the \emph{spectrum} and the \emph{eigenvalues} of a graph always refer to those of its Hermitian adjacency matrix.

Guo and Mohar \cite{mohar1} proved that $\rho(G_F^\sigma)\leq \rho(G)$ for any $F$ and $\sigma$.
Mohar \cite{mohar2} characterized the case when the equality is attained in the above inequality.
An interesting question is characterizing orientations for which the spectral radius is as small as possible.
 Chung and Graham \cite{Chung}, Griffiths \cite{Griffiths} considered quasi-randomness in digraphs which is related to the question.
Mohar \cite{mohar2} asked what is the minimum spectral radius taken over all orientations of a given graph.
Greaves, Mohar, and O \cite{Greaves} answered the question as follows:

\begin{thm}\cite{Greaves}\label{larg}
Let $G$ be a graph and let $\mu_G(x)$ be the matching polynomial of $G$.
Then there exists a complete orientation $\sigma$ of $G$ such that
$\lamax(G^\sigma) \le \rho(\mu_G)$, where $\rho(\mu_G)$ is the largest absolute value of the  roots of $\mu_G(x)$.
\end{thm}

As $G^\sigma$ is an oriented graph, the spectrum of $H(G^\sigma)$ is symmetric about the origin (see \cite{mohar1} or \cite{LL}).

\begin{cor}\cite{Greaves}\label{rho}
Let $G$ be a graph and let $\mu_G(x)$ be the matching polynomial of $G$.
Then there exists a complete orientation $\sigma$ of $G$ such that
$\rho(G^\sigma) \le \rho(\mu_G)$.
\end{cor}

In this paper we will discuss Mohar's question on mixed graphs with partial orientations.
Let $G$ be a connected graph and let $T$ be a spanning tree of $G$.
Write $G_{E(T)}^\sigma$ as $G_T^\sigma$ for short.
We prove that there exists a partial orientation $\sigma$ of $G$ with respect to $T$ such that the largest eigenvalue of $H(G_T^\sigma)$
 does not exceed the largest absolute value of the  roots of the matching polynomial of $G$.
We stress two points as follows.
(1) If $G$ is bipartite, then $H(G_T^\sigma)$ has a symmetric spectrum so that spectral radius of $G_T^\sigma$
 does not exceed the largest absolute value of the  roots of the matching polynomial of $G$.
(2) If $G$ contains even cycles, we can take a spanning tree $T$
such that under any partial orientation $\sigma$, $G_T^\sigma$ is not switching equivalent to an oriented graph or unoriented graph.
So our result is not a conclusion of Theorem \ref{larg} and Corollary \ref{rho}.

We follow the method of interlacing families of polynomials that was developed by
Marcus, Spielman, and Srivastava \cite{Marcus} in their seminal work on
the existence of infinite families of Ramanujan graphs, and are motivated by work of Greaves, Mohar, and O \cite{Greaves}.

Denote by $I$ an identity matrix. 
For a matrix or vector $M$ over $\C$, denote by $M^\top$ the transpose of $M$, $\overline{M}$ the conjugate of $M$, and $M^*:=\overline{M}^\top$.
Denote by $\mathbb{E} \v$ the expectation of a random variable $\v$.
For a positive integer $m$, denote $[m]:=\{1,\ldots,m\}$.

\section{Main results}

\subsection{Matching polynomials}

Let $G$ be a graph of order $n$.
A \emph{matching} in $G$ is a set of pairwise non-adjacent edges.
Let $m_k(G)$ be the number of matchings of $G$ consisting of $k$ edges and let $m_0=1$.
Heilmann and Lieb \cite{Heilmann} defined the {\it matching polynomial} of $G$ as
$$\mu_G(x) := \sum_{k\geq 0}(-1)^k m_k(G) x^{n-2k}.$$
Since $\mu_{G}(x)$ can be written as $x f(x^{2})$ or $f(x^{2})$ for some polynomial $f$, the roots of $\mu_{G}(x)$ are symmetric about the origin.
Denote by $\rho(\mu_G)$ the largest absolute value of the  roots of $\mu_G(x)$.
It is proved that $\mu_G(x)$ has only real roots \cite{Heilmann}.
So $\rho(\mu_G)$ is exactly the largest root of $\mu_G(x)$.

Godsil and Gutman \cite{GG} proved that the expected
characteristic polynomial over uniformly random signings of a graph is its matching polynomial.
Greaves, Mohar, and O \cite{Greaves} showed that the the expected
characteristic polynomial over uniformly random orientations of a graph is also equal to its matching polynomial.

Let $G=(V, E)$ be a connected graph and $T$ be a spanning tree of $G$.
Let $\Ori_T(G)$ be the set of all partial orientations of $G$ with respect to $T$.
 We show that the expected characteristic polynomial of $H(G_T^\sigma)$ over all uniformly random partial orientations $\sigma$ of $G$ with respect to $T$
  is equal to the matching polynomial of $G$.

\begin{thm} \label{random}
Let $G=(V, E)$ be a connected graph and $T$ be a spanning tree of $G$.
Then
$$\mathbb{E}_{\sigma \in \Ori_T(G)} \det (xI-H(G_T^{\sigma}))=\mu_{G}(x).$$
\end{thm}

\begin{proof}
Suppose $G$ has $n$ vertices.
Let $\mathfrak{S}(S)$ denote the set of permutations over a finite set $S$.
 By expanding the determinant as a sum over all permutations in $\mathfrak{S}(V)$, we have
\begin{align*}
\mathbb{E}_{\sigma \in \Ori_T(G)} \det(xI & - H(G_T^{\sigma}))= \mathbb{E}_{\sigma \in \Ori_T(G)}
\sum_{\pi \in \mathfrak{S}(V)} \sgn\pi \prod_{v \in V} (xI-H(G_T^{\sigma}))_{v, \pi(v)} \\
&=\sum_{k=0}^{n} (-1)^{k} x^{n-k} \sum_{S \subseteq V,\atop |S|=k}
\sum_{\pi \in \mathfrak{S}(S)}\sgn\pi \mathbb{E}_{\sigma \in \Ori_T(G)}  \prod_{v \in S} H(G_T^{\sigma})_{v,\pi(v)}.
\end{align*}

The entries of $H(G_T^{\sigma})=(h^\sigma_{uv})$ can be viewed as mutually independent random variables,
  except those constant entries $h_{u v}^{\sigma}=1$ for $\{u,v\} \in E(T)$,
  and pairs of $h_{u v}^{\sigma}$ and $h_{v u}^{\sigma}$ for $\{u,v\} \in E(G) \backslash E(T)$ which are
   inverse to each other (i.e. $h_{u v}^{\sigma} h_{v u}^{\sigma}=1$).
 Note that $\mathbb{E}_{\sigma \in \Ori_T(G)} h_{u v}^{\sigma}=1$ for every $\{u,v\} \in E(T)$,
 and  $\mathbb{E}_{\sigma \in \Ori_T(G)} h_{u v}^{\sigma}=0$ for every $\{u,v\} \in E(G) \backslash E(T)$.

For each $\pi \in \mathfrak{S}(S)$, if there exists a vertex $v \in S$ such that
   $\{v, \pi(v)\} \notin E(G)$, surely $\prod_{v \in S} H(G_T^{\sigma})_{v,\pi(v)}=0$.
So it suffices to consider those $\pi \in \mathfrak{S}(S)$ such that  $\{v, \pi(v)\} \in E(G)$ for all $v \in S$.
Let $\pi=\pi_1 \cdots \pi_t \in \mathfrak{S}(S)$ be a decomposition of $\pi$ into the product of disjoint cycles $\pi_1, \ldots, \pi_t$, where $t \ge 1$.
If $\pi$ contains a cycle say $\pi_1:=(v_{i_1} \ldots v_{i_t})$ of length $t$ at least $3$, then $G$ contains a cycle $C_{\pi_1}$ with edges
$\{v_{i_1},v_{i_2}\},\ldots, \{v_{i_t},v_{i_1}\}$.
As $T$ is a spanning tree of $G$, there exists at least one edge say $e_1:=\{v_{i_1},v_{i_2}\}$ of $C_{\pi_1}$ outside $T$.
So $e_1 $ is oriented by $\sigma$ in $G_T^\sigma$, and $\mathbb{E}_{\sigma \in \Ori_T(G)}h_{v_{i_1}v_{i_2}}^{\sigma}=0$,
which implies that in this situation
$$\mathbb{E}_{\sigma \in \Ori_T(G)}  \prod_{v \in S} H(G_T^{\sigma})_{v,\pi(v)}=0.$$
So it is enough to consider those $\pi$ which is a product of involutions $\pi_1, \ldots, \pi_t$.
Let $\pi_j=(v_{i_{2j-1}}, v_{i_{2j}})$ for $j \in [t]$.
Then the edges $\{v_{i_{2j-1}}, v_{i_{2j}}\}$ for $j \in [t]$ consist of a matching of size $t$ such that all vertices of $S$ are matched;
simply call $\pi$ a \emph{matching on $S$}.
In this case, $|S|$ is even, $t=|S|/2$, and for $j \in [t]$
$$\mathbb{E}_{\sigma \in \Ori_T(G)}h^\sigma_{v_{i_{2j-1}}, v_{i_{2j}}} h^\sigma_{v_{i_{2j}},v_{i_{2j-1}}}= \mathbb{E}_{\sigma \in \Ori_T(G)} 1=1.$$
So we have
$$\sgn\pi \mathbb{E}_{\sigma \in \Ori_T(G)}  \prod_{v \in S} H(G_T^{\sigma})_{v,\pi(v)}=(-1)^{|S|/2}.$$
By the above discussion, we get
$$
\mathbb{E}_{\sigma \in \Ori_T(G)}\det(xI-H(G_T^{\theta}))
= \sum_{k=0, \atop k ~{\rm even}}^n x^{n-k}\sum_{|S|=k} \sum_{{\rm matchings} \atop \pi {\rm on} S}(-1)^{k/2} \\
=\mu_G(x).
$$
\end{proof}

\subsection{Interlacing polynomials}

A univariate polynomial is called \emph{real-rooted} if all of its coefficients and roots are real.

\begin{defi}\cite{Marcus}
A real-rooted polynomial $g(x) = \prod_{j=1}^{n-1}(x - \alpha_j)$ \emph{interlaces} a real-rooted
polynomial $f(x) = \prod_{j=1}^{n}(x - \beta_j)$ if
$$\beta_1 \leq \alpha_1 \leq \beta_2 \leq \alpha_2 \leq \ldots \leq \alpha_{n-1} \leq \beta_n.$$
The polynomials $f_1, \ldots, f_k$ is said to have a {\it common interlacing} if
there is a single polynomial $g$ such that $g$ interlaces each of the $f_i$ for $i \in [k]$.
\end{defi}

\begin{defi}\cite{Marcus}
Let $S_1,\ldots,S_m$ be finite sets,
 and let $f_{s_1,\ldots,s_m}(x)$ be a real-rooted polynomial of degree $n$ with positive leading coefficient
 for every assignment $(s_1,\ldots,s_m) \in S_1 \times \cdots \times S_m$.
For a partial assignment $(s_1,\ldots, s_k) \in S_1 \times \cdots \times S_k$ with $k < m$, define
$$f_{s_1,\ldots, s_k}= \sum_{s_{k+1}\in S_{k+1},\ldots, s_m\in S_m}f_{s_1,\ldots, s_k,s_{k+1},\ldots, s_m}$$
as well as
$$f_\emptyset= \sum_{s_{1}\in S_{1},\ldots, s_m\in S_m}f_{s_1,\ldots, s_m}.$$
The polynomials $\{f_{s_1,\ldots, s_m}\}_{s_1,\ldots, s_m}$ is said to form an {\it interlacing family}
 if for all $k = 0,\ldots, m-1$ and all $(s_1,\ldots, s_k) \in S_1 \times \cdots \times S_k$,
 the polynomials $\{f_{s_1,\ldots, s_k,t}\}_{t\in S_{k+1}}$ have a common interlacing.
\end{defi}

\begin{lem}\label{fell}\cite{Dedieu,Fell,ChSe}
Let $f_{1}, \ldots, f_{k}$ be (univariate) polynomials of the same degree with positive leading coefficients.
Then $f_{1}, \ldots, f_{k}$ have a common interlacing
if and only if $\sum\limits_{i=1}^{k} \lambda_{i}f_{i}$ is real-rooted for all nonnegative $\lambda_{1}, \ldots, \lambda_{k}.$
\end{lem}

\begin{lem} \cite{Marcus} \label{interlacing}
Let $S_1,\ldots,S_m$ be finite sets, and let $\{f_{s_1,\ldots, s_m}\}$ be an interlacing family.
Then there exists some $(s_1,\ldots, s_m) \in S_1 \times \cdots \times S_m$
such that the largest root of $f_{s_1,\ldots, s_m}$ is at most the largest root of $f_\emptyset$.
\end{lem}

\begin{lem} \cite{Marcus, Greaves} \label{interlacing1}
If $a_1,\ldots, a_m, b_1,\ldots, b_m$ are vectors in $\mathbb{C}^n$, $D$ is a
Hermitian positive semidefinite matrix, and $p_1,\ldots, p_m$ are real numbers in $[0, 1]$.
Then the polynomial
$$\sum_{S \subseteq [m]}\left(\prod_{j\in S}p_j\right)\left(\prod_{j\notin S}(1-p_j)\right)
\det\left(xI+D+\sum\limits_{j\in S}a_ja_j^*+\sum\limits_{j\notin S}b_jb_j^*\right)$$
has only real roots.
\end{lem}

\subsection{Partial orientation}

Let $G=(V,E)$ be a connected graph and let $T$ be a spanning tree of $G$.
Suppose that $E(G)\backslash E(T)$ has $m$ edges, say $e_i=\{u_i,v_i\}$ for $i \in [m]$.
We specify a vertex say $u_i$ for each edge $e_i$ for $i \in [m]$.
Let $S_i =\{-1,1\}$ for $i \in [m]$.
Then the partial orientations $\sigma \in \Ori_T(G)$ are in bijective correspondence with the $m$-tuples
$(s_{1}, \ldots, s_{m}) \in S_{1} \times \cdots \times S_{m}$ by the rule $s_i=\sigma(u_i,v_i)$ for $i \in [m]$.
Under this correspondence, we define
$$f_{s_{1}, \ldots, s_{m}}(x) :=\det(xI-H(G_T^{\sigma})).$$

For a vertex $v \in V$, denote by $\ep_v \in \R^{V}$ a vector with entries indexed the vertices of $V$ such that
it has only one nonzero entry $1$ on the position indexed by $v$.

\begin{thm} \label{main1}
The polynomials $\{f_{s_{1}, \ldots, s_{m}}(x)\}$ form an interlacing family.
\end{thm}

\begin{proof}
By Lemma \ref{fell}, we only need to prove that for every $k=0, \ldots, m-1$, for all $(s_1,\ldots,s_k) \in \{-1,1\}^k$,
and for every $\lambda \in [0,1]$, the polynomial
$$\lambda f_{s_{1}, \ldots, s_{k}, 1}+(1-\lambda) f_{s_{1}, \ldots, s_{k},-1}=:P$$
is real-rooted.
Note that $P$ can be written as
$$P=\lambda \sum_{s_{k+2},\ldots, s_m\in \{\pm1\}}f_{s_1,\ldots,s_{k}, 1,s_{k+2},\ldots, s_m}+
(1-\lambda)\sum_{s_{k+2},\ldots, s_m\in \{\pm1\}}f_{s_1,\ldots,s_{k}, -1,s_{k+2},\ldots, s_m}.$$
We claim that we can rewrite $P$ as the form of Lemma \ref{interlacing1}
   by taking the following values for the constants $p_j$ and vectors $a_j, b_j$ ($j \in [m]$).

Set $$p_j= \left  \{
\begin{array}{ll}
(1+s_{j}) / 2 , & \hbox{~if~} 1 \leq j \leq k,\\
\lambda, & \hbox{~if~} j=k+1,\\
1/2, &  \hbox{~if~} k+2 \leq j \leq m.
\end{array}
\right.
$$

For each edge $e=\{u,v\}\in E(T)$, define a matrix $J_e:=(\ep_u-\ep_v)(\ep_u-\ep_v)^\top$.
Let $J_T:=\sum_{e \in E(T)} J_e$, which is positive semidefinite.
Define
$a_{j}=\ep_{u_{j}}+\i \ep_{v_{j}}$ and $b_{j}=\ep_{u_{j}}-\i \ep_{v_{j}}$ for $j \in [m]$.

For each partial orientation $\sigma \in \Ori_T(G)$ bijectively corresponding to $(s_1,\ldots,s_m) \in \{-1,1\}^m$,
if $S \subseteq [m]$ is the set of indices $j$ for which $s_{j}=1$, then
$$J_T+ \sum_{j \in S} a_{j} a_{j}^{*}+\sum_{j \notin S} b_{j} b_{j}^{*}=D-H(G_T^{\sigma}),$$
where $D$ is the diagonal matrix consisting of the degrees of vertices of $G$.
So we have
\begin{align*}
P&=\lambda \sum_{s_{k+2},\ldots, s_m\in \{\pm1\}}f_{s_1,\ldots,s_{k}, 1,s_{k+2},\ldots, s_m}+
(1-\lambda)\sum_{s_{k+2},\ldots, s_m\in \{\pm1\}}f_{s_1,\ldots,s_{k}, -1,s_{k+2},\ldots, s_m}\\
&=2^{m-k-1}\sum_{S \subseteq [m]}\left(\prod_{j\in S}p_j \right) \left(\prod_{j\notin S}(1-p_j) \right)
 \det \left(xI-D+J_T+\sum_{j\in S}a_ja_j^*+\sum_{j\notin S}b_jb_j^*\right).
\end{align*}
Let $\Delta$ be the maximum degree in $G$ and $y=x-\Delta.$ Then $P(x)=2^{m-k-1}Q(y)$, where
$$Q(y)=\sum_{S \subseteq [m]}\left(\prod_{j\in S} p_j \right) \left(\prod_{j\notin S}(1-p_j)\right)
  \det\left(yI+(\Delta I-D+J_T)+\sum_{j\in S}a_ja_j^*+\sum_{j\notin S}b_jb_j^*\right).$$
Since $\Delta I-D+J_T$ is positive semidefinite, Lemma \ref{interlacing1} indicates that $Q(y)$ has only real roots.
Hence $P(x)$ has only real roots and the result follows.
\end{proof}

Lemma \ref{interlacing} and Theorem \ref{main1} indicate that there exists $(s_1,\ldots, s_m) \in S_1 \times \cdots \times S_m$
(corresponding to a partial orientation $\sigma \in \Ori_T(G)$) such that the largest root of $f_{s_1,\ldots, s_m}(=\det(xI-H(G_T^{\sigma})))$
  is no more than the largest root of $f_\emptyset$, which is equal to $2^m \mathbb{E}_{\theta \in \Ori_T(G)} \det(xI-H(G_T^\sigma))=2^m\mu_{G}(x)$ by Theorem \ref{random}.
So we arrive at the main result of this paper.

\begin{thm}
Let $G$ be a connected graph and let $T$ be a spanning tree of $G$.
Let $\mu_G(x)$ be the matching polynomial of $G$.
Then there exists a partial orientation $\sigma$ of $G$ respect to $T$ such that
$\lamax(G^\sigma_T) \le \rho(\mu_G)$.
\end{thm}

If $G$ is a bipartite graph, then $H(G_T^\sigma)$ has a symmetric spectrum about the origin for any partial orientation $\sigma$ \cite{LL}.
So we have the following corollary.

\begin{cor}\label{bip}
Let $G$ be a connected bipartite graph and let $T$ be a spanning tree of $G$.
Let $\mu_G(x)$ be the matching polynomial of $G$.
Then there exists a partial orientation $\sigma$ of $G$ respect to $T$ such that
$\rho(G^\sigma_T) \le \rho(\mu_G)$.
\end{cor}

Mohar \cite{mohar2} introduced an operation on mixed graphs, called \emph{four-way switching}.
In matrix language, a four-way switching $s$ of a mixed graph $D$ is corresponding to a diagonal matrix $S$ with entries $S_{vv} \in \{\pm 1, \pm \i\}$
such that $S^{-1} H(D) S$ is a Hermitian adjacency matrix of a mixed graph denoted by $D^S$, where $D^S$ is obtained from $D$ by the four-switching $s$.
Call the matrix $S$ with the above property a \emph{switching matrix}.
The \emph{converse} of a mixed graph $D$, denoted by $D^\top$, is obtained from $D$ by reverse the orientation of each arc of $D$ \cite{mohar2}.
Note that $H(D^\top)=H(D)^\top=\overline{H(D)}$.

Two mixed graphs are called \emph{switching equivalent} if one can be obtained from another by a sequence of four-way switchings and/or taking converse \cite{mohar2}.
For a mixed graph $D$, applying to $D$ firstly by taking converse and then taking four-way switching by a switching matrix $S$, is equivalent to
applying to $D$ firstly by taking four-way switching by the switching matrix $\overline{S}(=S^{-1})$ and then taking converse,
 as $\overline{S}\overline{H(D)}S=\overline{SH(D)\overline{S}}$.
 Also, applying to $D$ by two four-way switchings is equivalent to one four-way switching
 as $\overline{S_2} (\overline{S_1} H(D) S_1) S_2= \overline{(S_1S_2)}H(D) (S_1S_2)$.
 So, a mixed graph $D_1$ is switching equivalent to $D_2$ if $D_2$ can be obtained from $D_1$ by firstly taking a four-way switchings and then at most one converse,
 namely $H(D_2)=S^{-1} H(D_1)S$ or $\overline{H(D_2)}=S^{-1} H(D_1)S$ for some switching matrix $S$.

\begin{lem}\label{four}
Let $G$ be a connected graph and let $T$ be a spanning tree of $G$,
and let $\sigma$ be a partial orientation of $G$ with respect to $T$

\begin{enumerate}
\item $G^\sigma_T$ is switching equivalent to $G$ if and only if $G=T$.

\item $G^\sigma_T$ is switching equivalent to an oriented graph of $G$ if and only if
for any even cycle $C$ of $G$, $T$ contains at most $|C|-2$ edges of $C$, where $|C|$ denotes the number of edges of $C$.
\end{enumerate}

\end{lem}

\begin{proof}
(1) Clearly the sufficiency holds by taking $I$ as the switching matrix.
Suppose $G^\sigma_T$ is switching equivalent to $G$.
As $\overline{H(G)}=H(G)$, there exists a switching matrix $S$ such that
\begin{equation}\label{swit}
S^{-1} H(G^\sigma_T) S=H(G).
\end{equation}
Fix a vertex $u \in V$, and without loss of generality assume that $S_{uu}=1$.
Let $v$ be a neighbor of $u$ in $T$.
Then, by comparing the $(u,v)$-entries of both sides of Eq. (\ref{swit}),
$S_{uu}^{-1} S_{vv}=1$, implying that $S_{vv}=1$.

Now let $w \ne u$ be an arbitrary vertex of $V$.
As $T$ is a spanning tree of $G$, there exists a path $P: u=u_0 u_1 \ldots u_p=w$ in $T$.
By the above discussion,
$$S_{uu}=S_{u_0 u_0}=S_{u_1 u_1} =\cdots =S_{u_p u_p}=S_{w w}=1.$$
So $S$ is an identity matrix, and $H(G^\sigma_T)=H(G)$.
The necessity follows by the definition.

(2) Fixing a vertex $u \in V$, let $V_1$ (respectively, $V_2$) be the set of vertices of $V$ with even distance (respectively, odd distance) to $u$ in $T$.
Then $V_1, V_2$ consist of a bipartition of $V$, where $u \in V_1$.

Assume that $G^\sigma_T$ is switching equivalent to an oriented graph $\vec{G}$.
As $\overline{H(\vec{G})}=H(\vec{G}^\top)$ and $\vec{G}^\top$ is still an oriented graph.
So we can assume that there exists a switching matrix $S$ such that
\begin{equation}\label{swit2}
S^{-1} H(G^\sigma_T) S=H(\vec{G}).
\end{equation}
Without loss of generality assume that $S_{uu}=1$.
Let $v \in V_2$ be a neighbor of $u$ in $T$.
Then, by comparing the $(u,v)$-entries of both sides of Eq. (\ref{swit2}),
$S_{uu}^{-1} S_{vv} \in \{-\i,\i\}$, implying that $S_{vv}\in \{-\i,\i\}$.
If $w \in V_1$ is a neighbor of $v$ in $T$ other than $u$, by Eq. (\ref{swit2}) we have
$S_{vv}^{-1} S_{ww} \in \{-\i,\i\}$, implying that $S_{ww} \in \{-1,1\}$.

Now let $z$ be an arbitrary vertex in $V_1$.
Then there exists a path $P: u=u_0 u_1 \ldots u_{2q}=z$ in $T$, where $u_{2i} \in V_1$ for $i=0,\ldots, q$, and $u_{2i+1} \in V_2$ for $i=0,\ldots,q-1$.
By the above discussion, we have $u_{2i} \in \{-1,1\}$ for $i=0,\ldots, q$, and $u_{2i+1} \in \{-\i,\i\}$ for $i=0,\ldots,q-1$.
So $S_{zz} \in \{-1,1\}$.
Similarly, $S_{zz} \in \{-\i,\i\}$ for an arbitrary vertex $z$ in $V_2$.

Let $C$ be an even cycle $C$ of $G$ of length $\ell$.
Assume to the contrary, $T$ contains $\ell-1$ edges of $C$.
Let $e=\{u_1,v_1\}$ be the edge of $C$ which is not lying on $T$, and is oriented by $\sigma$ to be an arc say $(u_1,v_1)$.
Then $e$ is an edge between $V_1$ and $V_2$, say $u_1 \in V_1$ and $v_1 \in V_2$.
Considering the $(u_1,v_1)$-entries of both sides of Eq. (\ref{swit2}),
we have
$$S_{u_1 u_1}^{-1} \cdot \i  \cdot S_{v_1 v_1} \in \{-\i,\i\},$$
which yields a contradiction as $ S_{u_1 u_1}\in \{-1,1\} $ and $S_{v_1 v_1} \in \{-\i,\i\}$.
So we prove the necessity.

Next assume that for any even cycle $C$ of $G$, $T$ contains  at most $|C|-2$ edges of $C$.
Then no edges of $E(G) \backslash E(T)$ lie between $V_1$ and $V_2$; otherwise $G$ would contain an even cycle $C$ such that $T$ contains $|C|-1$ edges of $C$.
Define a diagonal matrix $S$ with entries $S_{vv} \in \{-1,1\}$ if $v \in V_1$ and $S_{vv} \in \{-\i,\i\}$ if $v \in V_2$.
Then $S^{-1} H(G^\sigma_T) S$ is a Hermitian adjacency matrix of an oriented graph of $G$, where $S$ is the switching matrix.
\end{proof}

By Lemma \ref{four}, if $G$ contains no even cycles, then $G^\sigma_T$ is switching equivalent to an oriented graph.
So our result in this case is a conclusion of the results of Greaves et al. (Theorem \ref{larg} and Corollary \ref{rho}) as switching equivalence preserves the spectrum.

However, if $G$ contains an even cycle $C$ of length $\ell$ and $T$ is a spanning tree of $G$ containing $\ell -1$ edges of $C$,
then $G^\sigma_T$ cannot switching equivalent to oriented graphs.
So our result in this case is not within the conclusion of Greaves et al..
If $G$ is further bipartite, then by Corollary \ref{bip} we still have a result similar to Corollary \ref{rho}.
But, in general, Corollary \ref{bip} cannot hold for non-bipartite graphs.

\begin{exm}\label{ex1}
Let $G=D_5$, an undirected graph in Fig. \ref{4g}.
The matching polynomial of $G$ is $\mu_G(x)=x^4-5x^2+2$, whose spectral radius $\rho(\mu_G) \approx 2.136$.
If taking the path with edges $\{1,2\}, \{2,3\}, \{3,4\}$ as the spanning tree $T$ of $G$,
then we have two non-switching equivalent graphs $D_1$ and $D_2$ under partial orientations of $G$ with respect to $T$.
By Lemma \ref{four}, neither $D_1$ nor $D_2$  is switching equivalent to an oriented graph of $G$.
If taking the star with edges $\{1,2\}, \{2,3\}, \{2,4\}$ as the spanning tree $T'$ of $G$,
then we have two non-switching equivalent graphs $D_3$ and $D_4$ under partial orientations of $G$ with respect to $T'$.
By Lemma \ref{four}, both $D_3$ and $D_4$  are switching equivalent to oriented graph of $G$, which have symmetric spectra.

The characteristic polynomials of $D_i$ for $i \in [5]$ are respectively
\begin{align*}
\varphi_{D_1}(x)&=x^4-5x^2+2x+2, ~~~ \varphi_{D_2}(x)= x^4-5x^2-2x+2,\\
\varphi_{D_3}(x)&= x^4-5x^2+4, ~~~ \varphi_{D_4}(x) = x^4-5x^2, ~~~\varphi_{D_5}(x) = x^4-5x^2-4x,
\end{align*}
whose largest eigenvalues and least eigenvalues are respectively
\begin{align*}
 \lamax(D_1)& = -\lamin(D_2) \approx 1.814,~~~  \lamin(D_1) = - \lamax(D_2)\approx -2.343, \\
\rho(D_3)& =  2,  ~~~ \rho(D_4)  \approx 2.236,   ~~~ \lamax(D_5)  \approx 2.562, ~~~  \lamin(D_5) \approx -1.562.
\end{align*}
We have $\lamax(D_1) < \rho(\mu_G)$ but $\rho(D_1) \nleq \rho(\mu_G)$,
and $\rho(D_3) <  \rho(\mu_G)$.

\begin{figure}
\centering
   \setlength{\unitlength}{1bp}%
  \begin{picture}(355.49, 97.51)(0,0)
  \put(0,0){\includegraphics{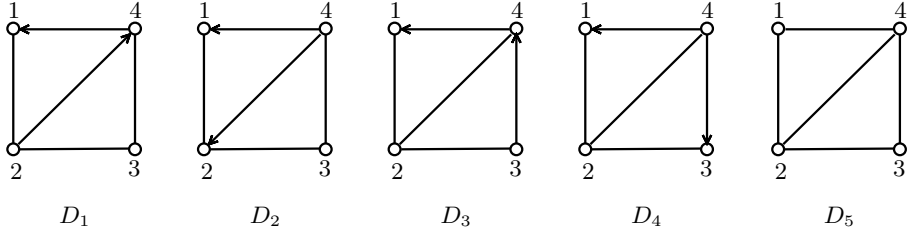}}
  \put(12.47,84.73){\fontsize{9.10}{10.93}\selectfont \makebox[0pt]{$1$}}
  \put(57.86,84.73){\fontsize{9.10}{10.93}\selectfont \makebox[0pt]{$4$}}
  \put(57.31,25.13){\fontsize{9.10}{10.93}\selectfont \makebox[0pt]{$3$}}
  \put(13.29,24.03){\fontsize{9.10}{10.93}\selectfont \makebox[0pt]{$2$}}
  \put(83.56,84.73){\fontsize{9.10}{10.93}\selectfont \makebox[0pt]{$1$}}
  \put(128.94,84.73){\fontsize{9.10}{10.93}\selectfont \makebox[0pt]{$4$}}
  \put(128.39,25.13){\fontsize{9.10}{10.93}\selectfont \makebox[0pt]{$3$}}
  \put(84.38,24.03){\fontsize{9.10}{10.93}\selectfont \makebox[0pt]{$2$}}
  \put(154.64,84.73){\fontsize{9.10}{10.93}\selectfont \makebox[0pt]{$1$}}
  \put(200.02,84.73){\fontsize{9.10}{10.93}\selectfont \makebox[0pt]{$4$}}
  \put(199.47,25.13){\fontsize{9.10}{10.93}\selectfont \makebox[0pt]{$3$}}
  \put(155.46,24.03){\fontsize{9.10}{10.93}\selectfont \makebox[0pt]{$2$}}
  \put(225.72,84.73){\fontsize{9.10}{10.93}\selectfont \makebox[0pt]{$1$}}
  \put(271.10,84.73){\fontsize{9.10}{10.93}\selectfont \makebox[0pt]{$4$}}
  \put(270.56,25.13){\fontsize{9.10}{10.93}\selectfont \makebox[0pt]{$3$}}
  \put(226.54,24.03){\fontsize{9.10}{10.93}\selectfont \makebox[0pt]{$2$}}
  \put(35.16,7.63){\fontsize{9.10}{10.93}\selectfont \makebox[0pt]{$D_1$}}
  \put(106.25,7.63){\fontsize{9.10}{10.93}\selectfont \makebox[0pt]{$D_2$}}
  \put(177.33,7.63){\fontsize{9.10}{10.93}\selectfont \makebox[0pt]{$D_3$}}
  \put(248.41,7.63){\fontsize{9.10}{10.93}\selectfont \makebox[0pt]{$D_4$}}
  \put(297.63,84.73){\fontsize{9.10}{10.93}\selectfont \makebox[0pt]{$1$}}
  \put(343.01,84.73){\fontsize{9.10}{10.93}\selectfont \makebox[0pt]{$4$}}
  \put(342.47,25.13){\fontsize{9.10}{10.93}\selectfont \makebox[0pt]{$3$}}
  \put(298.45,24.03){\fontsize{9.10}{10.93}\selectfont \makebox[0pt]{$2$}}
  \put(320.32,7.63){\fontsize{9.10}{10.93}\selectfont \makebox[0pt]{$D_5$}}
  \end{picture}%
\caption{Five non-switching equivalent non-bipartite mixed graphs}\label{4g}
\end{figure}
\end{exm}

\begin{exm}\label{ex2}
Let $G=H_3$, an undirected graph in Fig. \ref{4h}.
The matching polynomial of $G$ is $\mu_G(x)=x^4-4x^2+2$, whose spectral radius $\rho(\mu_G) \approx 1.848$.
Taking the path with edges $\{1,2\}, \{2,3\}, \{3,4\}$ as the spanning tree $T$ of $G$,
then we have only one graph under switching equivalence, namely $H_1$ of Fig. \ref{4h}, under partial orientation of $G$ with respect to $T$.
By Lemma \ref{four}, $H_1$  is not switching equivalent to an oriented graph of $G$.
Any oriented graph of $G$ is switching equivalent to $H_2$ or $H_3$ of Fig. \ref{4h}.

The characteristic polynomials of $H_i$ for $i \in [3]$ are respectively
$$
\varphi_{H_1}(x)=\mu_G(x)=x^4-4x^2+2, \varphi_{H_2}(x)= x^4-4x^2+4, \varphi_{H_3}(x)= x^4-4x^2.
$$
whose spectral radii hold
$$ \rho(H_2) \approx 1.414 < \rho(H_1)= \rho(\mu_G) \approx 1.848 < \rho(H_3)=2.$$

\begin{figure}
\centering
  \setlength{\unitlength}{1bp}%
  \begin{picture}(212.49, 97.51)(0,0)
  \put(0,0){\includegraphics{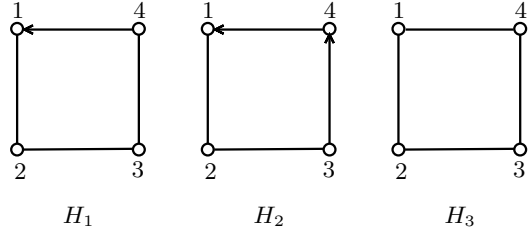}}
  \put(12.47,84.73){\fontsize{9.10}{10.93}\selectfont \makebox[0pt]{$1$}}
  \put(57.86,84.73){\fontsize{9.10}{10.93}\selectfont \makebox[0pt]{$4$}}
  \put(57.31,25.13){\fontsize{9.10}{10.93}\selectfont \makebox[0pt]{$3$}}
  \put(13.29,24.03){\fontsize{9.10}{10.93}\selectfont \makebox[0pt]{$2$}}
  \put(83.56,84.73){\fontsize{9.10}{10.93}\selectfont \makebox[0pt]{$1$}}
  \put(128.94,84.73){\fontsize{9.10}{10.93}\selectfont \makebox[0pt]{$4$}}
  \put(128.39,25.13){\fontsize{9.10}{10.93}\selectfont \makebox[0pt]{$3$}}
  \put(84.38,24.03){\fontsize{9.10}{10.93}\selectfont \makebox[0pt]{$2$}}
  \put(154.64,84.73){\fontsize{9.10}{10.93}\selectfont \makebox[0pt]{$1$}}
  \put(200.02,84.73){\fontsize{9.10}{10.93}\selectfont \makebox[0pt]{$4$}}
  \put(199.47,25.13){\fontsize{9.10}{10.93}\selectfont \makebox[0pt]{$3$}}
  \put(155.46,24.03){\fontsize{9.10}{10.93}\selectfont \makebox[0pt]{$2$}}
  \put(35.16,7.63){\fontsize{9.10}{10.93}\selectfont \makebox[0pt]{$H_1$}}
  \put(106.25,7.63){\fontsize{9.10}{10.93}\selectfont \makebox[0pt]{$H_2$}}
  \put(177.33,7.63){\fontsize{9.10}{10.93}\selectfont \makebox[0pt]{$H_3$}}
  \end{picture}%
\caption{Three non-switching equivalent bipartite mixed graphs}\label{4h}
\end{figure}

\end{exm}

From Example \ref{ex1} and Example \ref{ex2}, we wonder that
among all orientations of a graph, the minimum spectral radius of Hermitian adjacency matrix is attained at an oriented graph.

\end{document}